\documentclass[12pt]{article}

\usepackage{amsmath, amssymb, amsthm, mathrsfs}
\usepackage[margin=1in]{geometry}
\usepackage{graphicx}
\usepackage{enumitem}
\usepackage{fancyhdr}
\usepackage{lastpage}
\usepackage{scrextend}
\usepackage{enumitem}
\usepackage{tikz-cd}
\usepackage{xcolor}
\usepackage[doi=false,isbn=false,url=false,eprint=true]{biblatex}
\usepackage{hyperref}

\newcommand{\Z}{\mathbb{Z}}

\renewcommand{\P}{\mathbb{P}}
\newcommand{\B}{\mathbb{B}}

\newcommand{\dom}{\operatorname{dom}}

\newcommand{\id}{\operatorname{id}}

\newcommand{\restrict}{{\restriction}}

\newcommand{\Ext}{\operatorname{Ext}}

\newcommand{\Hom}{\operatorname{Hom}}
\newcommand{\PHom}{\operatorname{PHom}}

\newcommand{\Add}{\operatorname{Add}}

\newcommand{\ab}[1]{\mathopen{}\left\langle #1 \right\rangle\mathclose{}}

\newcommand{\set}[2]{\left\{ #1 \,:\, #2 \right\}}

\renewbibmacro{in:}{}

\setlist{topsep=0pt, partopsep=0pt, itemsep=0pt}
\setlist[enumerate]{label=(\arabic*)}

\allowdisplaybreaks
\theoremstyle{definition}
\newtheorem{definition}{Definition}[section]
\newtheorem{theorem}[definition]{Theorem}
\newtheorem{proposition}[definition]{Proposition}
\newtheorem{lemma}[definition]{Lemma}

\newtheorem{corollary}[definition]{Corollary}

\title{Condensed Whitehead Problem}
\author{Spencer Unger and Jerry Wei}
\date{}

\begin{document}
\maketitle
\sloppy

\begin{abstract}
    The Whitehead problem, which asks whether every abelian group \( A \) satisfying \( \Ext^1(A,\Z) = 0 \) is free, is independent of ZFC. However, this problem has an analogue in condensed mathematics that can be answered affirmatively. In this note, we give a new proof that an abelian group \( A \) of size \( \kappa \) is free if and only if \( \underline{\Ext}^1(\underline{A}, \underline{\Z})(S_\kappa) = 0 \), where \( S_\kappa \) is the Stone space of the Boolean completion of \( \operatorname{Add}(\omega,\kappa) \). 
\end{abstract}

\section{Introduction}

In recent years, Dustin Clausen and Peter Scholze developed a new framework, called condensed mathematics \cite{ClausenScholze2019}, for doing homological algebra on algebraic structures equipped with topology. Classical categories involving such structures are often poorly behaved; for instance, the category of topological abelian groups is not an abelian category. Condensed mathematics is intended to fix this issue by enriching classical categories with ``condensed objects''. Under this framework, classical objects such as sets and abelian groups have condensed counterparts, and thus many classical questions can be naturally reformulated in the condensed setting. In this note, we will consider a typical example --- the Whitehead problem. 

An abelian group is \emph{Whitehead} if \( \Ext^1(A,\Z) = 0 \), i.e., every short exact sequence of abelian groups \( 0 \to \Z \to B \to A \to 0 \) splits. This is a natural weakening of freeness -- recall that an abelian group \( A \) is free if and only if \( \Ext^1(A,C) = 0 \) for every abelian group \( C \). In the 1950s, J.~H.~C.~Whitehead asked whether every Whitehead group is free. Karl Stein gave an affirmative answer for countable groups \cite{Stein}, but the problem for uncountable groups remained open for twenty years. Surprisingly, in 1974, Saharon Shelah proved that Whitehead's problem is independent of ZFC, i.e., its answer cannot be decided by the usual axioms of set theory \cite{ShelahWhiteheadProblem}.

We now explain how to generalize the Whitehead problem in the condensed setting. A \emph{condensed abelian group} is a contravariant functor from the category of profinite spaces to the category of abelian groups satisfying certain sheafiness conditions. Every topological abelian group \( A \) induces a condensed abelian group \( \underline{A} \), given by \( \underline{A}(X) = C(X,A) \) being the abelian group of continuous functions from \( X \) to \( A \) for every profinite space \( X \). If the topology on \( A \) is not specified, we use the discrete topology by default. In the category of condensed abelian groups, there turns out to be a functor \( \underline{\Ext}^1 \) corresponding to the classical \( \Ext^1 \). This leads to the following definition:

\begin{definition}
    An abelian group \( A \) is \emph{\( X \)-condensed Whitehead} for a profinite space \( X \) if \( \underline{\Ext}^1(\underline{A}, \underline{\Z})(X) = 0 \). We say that \( A \) is \emph{condensed Whitehead} if it is \( X \)-condensed Whitehead for every profinite space \( X \).
\end{definition}

In analogy with the Whitehead problem, it is natural to ask: Is every condensed Whitehead group free? This condensed Whitehead problem turned out to be no longer independent of ZFC. In 2020, Clausen-Scholze answered the condensed Whitehead problem affirmatively using a structural
analysis of the category of condensed abelian groups \cite{ClausenScholzeYouTube}. Later in 2024, Jeffrey Bergfalk, Chris Lambie-Hanson and Jan \v{S}aroch gave a more elementary and constructive proof of a stronger result \cite{Lambie-Hanson}: 

\begin{theorem}\label{thm1}
    Let \( A \) be an abelian group of size \( \kappa \). Let \( \Add(\omega,\kappa) \) be the forcing for adding \( \kappa \)-many Cohen reals, and let \( S_\kappa \) be the Stone space of its Boolean completion. Then the following are equivalent: 
    \begin{enumerate}
        \item \( A \) is \( S_\kappa \)-condensed Whitehead;
        \item \( A \) is free.
    \end{enumerate}
\end{theorem}

In this note, we present a new proof of Theorem \ref{thm1}. The proof is based on Eklof and Mekler's proof \cite[\nopp XII.1.6]{EklofMekler1990} of the consistency of the positive answer to the classical Whitehead problem, which uses some version of a combinatorial principle called weak diamond. Our proof does not require understanding condensed mathematics, and is accessible to anyone with basic knowledge in algebra, topology, and set theory. We remark that although the theorem is motivated by forcing, the proof requires no familiarity with forcing beyond basic facts about partial orders and Boolean algebras.

Forcing appears naturally in solving the condensed Whitehead problem. Recall that every forcing notion is equivalent to a Boolean algebra. By Stone duality, the Stone spaces of Boolean algebras are precisely the profinite spaces, which are fundamental building blocks in condensed mathematics. Building on this duality, we obtain:

\begin{corollary}\label{cor:forcing-characterization}
    In Theorem \ref{thm1}, clauses (1) and (2) are also equivalent to the following:
    \begin{enumerate}
    \item[(3)] \( \Add(\omega,\kappa) \) forces \( A \) to be Whitehead.
    \end{enumerate}
\end{corollary}

In Section \ref{sec:condensed-whiteheadness-and-forcing}, we study characterizations of condensed Whiteheadness and its variants. In Subsection \ref{sec:forced-whiteheadness}, we characterize the property of being forced to be Whitehead, which will be used to prove Corollary \ref{cor:forcing-characterization}. Readers who are unfamiliar with forcing may safely skip this subsection, as the proof of Theorem \ref{thm1} does not depend on it. In Section \ref{sec:filtration}, we provide an expository discussion about filtrations of groups, which will be essential for the proof of the main theorem. Finally, in Section \ref{sec:main-theorem}, we use a variant of condensed Whiteheadness to prove a stronger theorem which implies both Theorem \ref{thm1} and Corollary \ref{cor:forcing-characterization}.

\section{Condensed Whiteheadness and Forcing}\label{sec:condensed-whiteheadness-and-forcing}   

\subsection{Condensed Whiteheadness for extremally disconnected spaces}

We will use an equivalent definition of condensed Whiteheadness, which does not require understanding \( \underline{\Ext}^1 \). Recall that a \emph{free resolution} of an abelian group \( A \) is a short exact sequence \( 0 \to K \to F \to A \to 0 \), where \( K,F \) are free abelian groups. We shall identify \( K \) as a subgroup of \( F \). By considering the induced long exact sequence of the functor \( \Hom(-,\Z) \), it is not hard to show that \( A \) is Whitehead if and only if every \( \varphi \in \Hom(K,\Z) \) extends to some \( \psi \in \Hom(F,\Z) \). 

We consider \( \Hom(K,\Z) \) as a subspace of the product space \( \Z^K \), where \( \Z \) is equipped with the discrete topology. Likewise, we consider \( \Hom(F,\Z) \) as a subspace of \( \Z^F \). 

\begin{definition}
    Let \( X \) be a topological space. Let \( A \) be an abelian group with \( 0 \to K \to F \to A \to 0 \). We say that \( A \) is \emph{\( X \)-parametrized Whitehead} if every continuous map \( \varphi : X \to \Hom(K,\Z) \) extends pointwise to some continuous map \( \psi : X \to \Hom(F,\Z) \), i.e., \( \psi(x) \restrict K = \varphi(x) \) for all \( x \in X \). 
\end{definition}

This definition is independent of the choice of free resolution. By \cite[\nopp Section 3]{Lambie-Hanson}, this characterizes \( X \)-condensed Whiteheadness when \( X \) is \emph{extremally disconnected}, i.e., the closure of every open subset of \( X \) is open:

\begin{proposition}\label{prop:extremally-disconnected}
    Let \( X \) be an extremally disconnected compact Hausdorff space. Then \( A \) is \( X \)-condensed Whitehead if and only if it is \( X \)-parametrized Whitehead. \qed
\end{proposition}

By Stone duality, extremally disconnected compact Hausdorff spaces are exactly the Stone spaces of complete Boolean algebras. The Stone space of a Boolean algebra is the space of ultrafilters on this algebra. We generalize this notion to arbitrary posets using maximal filters. 




\subsection{Maximal filter spaces and Boolean completions} 

The \emph{maximal filter space} \( S(\P) \) of a poset \( \P \) is the space of maximal filters on \( \P \), equipped with the topology generated by the basic open sets \( N_p = \set{x \in S(\P)}{p \in x} \) for \( p \in \P \). We are mostly interested in separative posets, since if \( \P \) is separative, then \( p \leq q \) if and only if \( N_p \subseteq N_q \). If \( \B \) is a Boolean algebra, then \( S(\B) \) is a profinite space, and is commonly known as the \emph{Stone space} of \( \B \); moreover, if \( \B \) is complete, then \( S(\B) \) is extremally disconnected. 

Recall that for every poset \( \P \), there is a complete Boolean algebra \( \B \) and a dense embedding \( i : \P \to \B \), known as the \emph{Boolean completion} of \( \P \). We will establish a correspondence between homomorphisms parametrized by \( S(\P) \) and those by \( S(\B) \).

For abelian groups \( A,B \), let \( \PHom(A,B) \) denote the set of homomorphisms from a subgroup of \( A \) to \( B \). We may regard elements of \( \PHom(A,B) \) as functions from \( A \) to \( B \cup \{ \star \} \) which take value \( \star \) whenever they are undefined. We equip \( B \cup \{ \star \} \) with the weakest topology such that each singleton \( \{ b \} \subseteq B \) is open, and consider \( \PHom(A,B) \) as a subspace of the product space \( (B \cup \{ \star \})^A \). Hence, a map \( \varphi : X \to \PHom(A,B) \) is continuous if and only if \( \set{x \in X}{\varphi(x)(a) = b} \) is open for every \( a \in A \) and \( b \in B \). We do not require \( \set{x \in X}{\varphi(x)(a) = \star} \) to be open. 

\begin{definition}
    Let \( \varphi : X \to \PHom(A,B) \). The \emph{domain} of \( \varphi \) at \( a \in A \) is the set 
    \[ \dom(\varphi,a) = \set{x \in X}{\varphi(x)(a) \text{ is defined}} \]
    We say that \( \varphi \) is \emph{full} if \( \dom(\varphi,a) \) contains an open dense set for all \( a \in A \). 
\end{definition}

Let \( \varphi,\psi : X \to \PHom(A,B) \) be full. For \( H \subseteq A \), we say \emph{\( \varphi =^* \psi \) on \( H \)} if for all \( a \in H \), there is an open dense set of \( x \in X \) such that \( \varphi(x)(a) = \psi(x)(a) \). We simply write \emph{\( \varphi =^* \psi \)} if \( \varphi =^* \psi \) on \( A \). We define \( \varphi + \psi \) to be the full map given by \( (\varphi + \psi)(x)(a) = \varphi(x)(a) + \psi(x)(a) \) whenever \( \varphi(x)(a) \) and \( \psi(x)(a) \) are defined, and \( -\varphi \) is defined similarly. For an open subset \( N \subseteq X \), we write \( \varphi[N](a) = b \) if \( \varphi(x)(a) = b \) whenever \( x \in N \) and \( \varphi(x)(a) \) is defined.

Note that if \( \mathbb{D} \subseteq \P \) is dense, then \( \bigcup_{p \in \mathbb{D}} N_p \) is open dense in \( S(\P) \). Conversely, if \( D \subseteq S(\P) \) is open dense, then \( \set{p \in \P}{N_p \subseteq D} \)  is open dense in \( \P \).

\begin{definition}
    A map \( \varphi : X \to \PHom(A,B) \) is \emph{bounded} if for all \( a \in A \), the set \( \set{b \in B}{\exists x \in X (\varphi(x)(a) = b)} \) is finite. 
\end{definition}

\begin{theorem}[First Correspondence Theorem]\label{thm:bounded-correspondence}
    Let \( \P \) be a separative poset and let \( i : \P \to \B \) be its Boolean completion. Then there is a correspondence
    \[\begin{tikzcd}
        {\{ \text{bounded full continuous maps \( \varphi : S(\P) \to \PHom(A,B) \)} \}} \\
        {\{ \text{continuous maps \( \varphi' : S(\B) \to \Hom(A,B) \)} \}}
        \arrow["\alpha"', shift right, from=1-1, to=2-1]
        \arrow["\beta"', shift right, from=2-1, to=1-1]
    \end{tikzcd}\]
    such that if \( \varphi, \varphi' \) are corresponding (either \( \alpha(\varphi) = \varphi' \) or \( \beta(\varphi') = \varphi \)), then for \( a \in A \), \( b \in B \), \( x \in \dom(\varphi,a) \), and \( U \in S(\B) \), we have
    \begin{align}
        \varphi(x)(a) = b &\iff \exists p \in x (\varphi'[N_{i(p)}](a) = b), \\
        \bigvee \set{i(p)}{\varphi[N_p](a) = b} \in U &\iff \varphi'(U)(a) = b.
    \end{align}
    Moreover, if \( \psi, \psi' \) are corresponding, then
    \[ \varphi =^* \psi \iff \varphi' = \psi'. \]
    \begin{proof}
        We first define the correspondence map \( \alpha \). Let \( \varphi : S(\P) \to \PHom(A,B) \) be a bounded full continuous map. Fix \( a \in A \) and \( U \in S(\B) \). We define \( \varphi'(U)(a) \) as follows. Let 
        \[ F = \set{b \in B}{\exists x \in S(\P) (\varphi(x)(a) = b)} = \set{b \in B}{\exists p \in \P (\varphi[N_p](a) = b)}. \]
        Since \( \varphi \) is bounded, \( F \) is finite. For each \( b \in F \), let \( c_b = \bigvee \set{i(p)}{\varphi[N_p](a) = b} \). Then \( \set{c_b}{b \in F} \) is a finite partition of \( \B \), so we can define \( \varphi'(U)(a) = b \) to be the unique \( b \in F \) such that \( c_b \in U \). 
        
        The equivalence (2) is obvious from the definition. We check that (1) holds. Let \( x \in \dom(\varphi,a) \). If \( \varphi(x)(a) = b \), then by continuity, there is \( p \in x \) such that \( \varphi[N_p](a) = b \). Then \( i(p) \leq c_b \), so for all \( U \in N_{i(p)} \), we have \( c_b \in U \) and thus \( \varphi'(U)(a) = b \). Conversely, assume there is \( p \in x \) such that \( \varphi'[N_{i(p)}](a) = b \). Suppose \( \varphi(x)(a) = b' \neq b \). Then there is \( p' \in x \) such that \( \varphi'[N_{i(p')}](a) = b' \). Since \( x \) is a filter, there is \( r \leq p,p' \). Then \( i(r) \leq i(p),i(p') \), so if \( U \in N_{i(r)} \), we have \( \varphi'(U)(a) = b \) and \( \varphi'(U)(a) = b' \), a contradiction.

        We now define the correspondence map \( \beta \). Let \( \varphi' : S(\B) \to \Hom(A,B) \) be given. Define \( \varphi : S(\P) \to \PHom(A,B) \) by
        \[ \varphi(x)(a) = b \iff \exists p \in x (\varphi'[N_{i(p)}](a) = b). \]
        Then \( \varphi \) is bounded since \( \set{\varphi'(U)(a)}{U \in S(\B)} \) is finite, being a continuous image of the compact space \( S(\B) \). 
        
        The equivalence (1) is obvious from the definition. We check that (2) holds. For \( b \in B \), let 
        \[ c_b = \bigvee \set{i(p)}{\varphi[N_p](a) = b}. \]
        We need to show that \( c_b \in U \) if and only if \( \varphi'(U)(a) = b \).
        If \( \varphi'(U)(a) = b \), then by continuity, there is \( c \in U \) such that \( \varphi'[N_c](a) = b \). We show that \( c \leq c_b \). Since \( i[\P] \) is dense in \( \B \), it suffices to show that \( i(p) \leq c_b \) for all \( i(p) \leq c \). If \( i(p) \leq c \), then \( \varphi'[N_{i(p)}](a) = b \) and thus \( \varphi[N_p](a) = b \), so \( i(p) \leq c_b \). Conversely, assume \( c_b \in U \). Suppose \( \varphi'(U)(a) = b' \neq b \). Then \( c_{b'} \in U \). But it is not hard to see that \( c_b \wedge c_{b'} = 0 \), contradicting that \( U \) is an ultrafilter. 

        Finally, we show that \( \varphi =^* \psi \) if and only if \( \varphi' = \psi' \). Assume \( \varphi =^* \psi \). We show that 
        \[ \bigvee \set{i(p)}{\varphi[N_p](a) = b} = \bigvee \set{i(p)}{\psi[N_p](a) = b}. \]
        We first show that the left hand side is less than or equal to the right hand side. Let \( q \) be such that \( \varphi[N_q](a) = b \). Since \( \varphi =^* \psi \), there is an open dense subset \( D \subseteq N_q \) such that \( \psi[D](a) = b \). Hence, \( \set{p \leq q}{N_p \subseteq D} \) is dense below \( q \), so 
        \[ i(q) = \bigvee \set{i(p)}{N_p \subseteq D} \leq \bigvee \set{i(p)}{\psi[N_p](a) = b}. \]
        Taking the supremum over all \( q \)'s such that \( \varphi[N_q](a) = b \) gives the desired inequality. The other inequality is similar. Then it follows from (2) that \( \varphi' = \psi' \). Conversely, if \( \varphi' = \psi' \), then \( \varphi(x)(a) = \psi(x)(a) \) for all \( x \in \dom(\varphi,a) \cap \dom(\psi,a) \) by (1). Since \( \dom(\varphi,a) \cap \dom(\psi,a) \) is open dense, we have \( \varphi =^* \psi \). 
    \end{proof}
\end{theorem}

Let \( K \) be a subgroup of \( F \). We say that \( \psi : X \to \PHom(F,B) \) \emph{almost extends} \( \varphi : X \to \PHom(K,B) \) if for all \( a \in K \), there is an open dense set of \( x \in X \) such that \( \psi(x)(a) = \varphi(x)(a) \). The correspondence theorem immediately yields:

\begin{corollary}
    Let \( K \) be a subgroup of \( F \). Let \( \P \) be a separative poset. Let \( \varphi : S(\P) \to \PHom(K,B) \) and \( \psi : S(\P) \to \PHom(F,B) \) be bounded full continuous maps. Using the notation in the correspondence theorem, we have
    \[ \text{\( \psi \) almost extends \( \varphi \)} \iff \text{\( \psi' \) extends \( \varphi' \)} \]
    \qed
\end{corollary}

This allows us to characterize \( S(\B) \)-condensed Whiteheadness in terms of \( S(\P) \):

\begin{theorem}\label{thm:characterization}
    Let \( A \) be an abelian group with free resolution \( 0 \to K \to F \to A \to 0 \). Let \( \P \) be a separative poset and let \( \B \) be its Boolean completion. Then the following are equivalent:
    \begin{enumerate}
        \item \( A \) is \( S(\B) \)-condensed Whitehead;
        \item Every bounded full continuous map \( \varphi : S(\P) \to \PHom(K,\Z) \) almost extends to some bounded full continuous map \( \psi : S(\P) \to \PHom(F,\Z) \). 
    \end{enumerate} \qed
\end{theorem}

This will be useful later when we prove Theorem \ref{thm1}, since the maximal filter space of \( \Add(\omega,\kappa) \) is homeomorphic to \( 2^\kappa \), which is easier to work with than the Stone space of its Boolean completion.

\subsection{Weak \( 2^\kappa \)-parametrized Whiteheadness}\label{sec:weak-parametrized}

Let \( \P_\kappa = \Add(\omega,\kappa) \) be the forcing for adding \( \kappa \)-many Cohen reals, i.e., 
\[ \P_\kappa = \{ \text{finite partial functions \( p : \kappa \to 2 \)} \} \]
ordered by extension. Its maximal filter space is homeomorphic to \( 2^\kappa \) with the product topology via the identification \( F \mapsto \bigcup F \). For \( p \in \P_\kappa \), we shall also write \( N_p = \set{x \in 2^\kappa}{p \subseteq x} \) for the basic open sets of \( 2^\kappa \).


We will use the following weakening of \( 2^\kappa \)-parametrized Whiteheadness. 

\begin{definition}
    Let \( A \) be an abelian group with free resolution \( 0 \to K \to F \to A \to 0 \). We say that \( A \) is \emph{weakly \( 2^\kappa \)-parametrized Whitehead} if for all continuous map \( \varphi : 2^\kappa \to \Hom(K,\Z) \), there is \( \lambda \geq \kappa \) and a full continuous map \( \psi : 2^\lambda \to \PHom(F,\Z) \) almost extending \( \varphi \) in the sense that for all \( a \in K \), there is an open dense set of \( x \in 2^\lambda \) such that \( \psi(x)(a) = \varphi(x \restrict \kappa)(a) \).
\end{definition}

This definition is independent of the choice of free resolution. Suppose \( A \) is weakly \( 2^\kappa \)-parametrized Whitehead with respect to the free resolution \( 0 \to K \to F \to A \to 0 \). Let \( 0 \to K' \to F' \to A \to 0 \) be another free resolution of \( A \). By basic homological algebra, we have the following diagram
\[\begin{tikzcd}
    0 & {K'} & {F'} & A & 0 \\
    0 & K & F & A & 0 \\
    0 & {K'} & {F'} & A & 0
    \arrow[from=1-1, to=1-2]
    \arrow[from=1-2, to=1-3]
    \arrow["{f \restrict K'}"', from=1-2, to=2-2]
    \arrow[from=1-3, to=1-4]
    \arrow["{f}", from=1-3, to=2-3]
    \arrow["h"'{pos=0.3}, from=1-3, to=3-2]
    \arrow[from=1-4, to=1-5]
    \arrow["{\mathrm{id}}", from=1-4, to=2-4]
    \arrow[from=1-4, to=3-3]
    \arrow[from=2-1, to=2-2]
    \arrow[from=2-2, to=2-3]
    \arrow["{g \restrict K}"', from=2-2, to=3-2]
    \arrow[from=2-3, to=2-4]
    \arrow["{g}"', from=2-3, to=3-3]
    \arrow[from=2-4, to=2-5]
    \arrow["{\mathrm{id}}", from=2-4, to=3-4]
    \arrow[from=3-1, to=3-2]
    \arrow[from=3-2, to=3-3]
    \arrow[from=3-3, to=3-4]
    \arrow[from=3-4, to=3-5]
\end{tikzcd}\]
where the squares commute and the slanted maps give a chain homotopy from \( g \circ f \) to \( \id \). In particular, \( g(f(a)) - a = h(a) \) for all \( a \in K' \). Now given \( \varphi' : 2^\kappa \to \Hom(K',\Z) \), let \( \varphi : 2^\kappa \to \Hom(K,\Z) \) be defined by \( \varphi(x)(a) = \varphi'(x)(g(a)) \). Let \( \psi : 2^\lambda \to \PHom(F,\Z) \) be an almost extension of \( \varphi \). Let \( \psi' : 2^\lambda \to \PHom(F',\Z) \) be given by \( \psi'(x)(a) = \psi(x)(f(a)) - \varphi'(x \restrict \kappa)(h(a)) \) whenever it is defined. It is not hard to see that \( \psi' \) is full, and for \( a \in K' \), if \( \psi(x)(f(a)) = \varphi(x \restrict \kappa)(f(a)) \), then 
\[ \begin{split}
    \psi'(x)(a) &= \psi(x)(f(a)) - \varphi'(x \restrict \kappa)(h(a)) \\
    &= \varphi(x \restrict \kappa)(f(a)) - \varphi'(x \restrict \kappa)(h(a)) \\
    &= \varphi'(x \restrict \kappa)((g \circ f)(a)) - \varphi'(x \restrict \kappa)(h(a)) \\
    &= \varphi'(x \restrict \kappa)(a).
\end{split} \]

Since \( 2^\kappa \) is compact, every continuous map \( \varphi : 2^\kappa \to \Hom(A,B) \) is also a bounded full continuous map \( 2^\kappa \to \PHom(A,B) \). This fact together with Theorem \ref{thm:characterization} immediately implies the following lemma:

\begin{lemma}\label{lem:condensed_whitehead_implies_weakly_parametrized}
    Let \( S_\kappa \) be the Stone space of the Boolean completion of \( \P_\kappa \).
    Assume \( A \) satisfies either one of the following conditions:
    \begin{enumerate}
        \item \( A \) is \( S_\kappa \)-condensed Whitehead, or
        \item \( A \) is \( 2^\kappa \)-parametrized Whitehead.
    \end{enumerate} 
    Then \( A \) is weakly \( 2^\kappa \)-parametrized Whitehead. \qed
\end{lemma}

\begin{lemma}\label{lem:weakly_parametrized_implies_smaller_weakly_parametrized}
    If \( A \) is weakly \( 2^\kappa \)-parametrized Whitehead, then it is weakly \( 2^\gamma \)-parametrized Whitehead for all \( \gamma \leq \kappa \). \qed
\end{lemma}

\begin{lemma}
    If \( A \) is weakly \( 2^\kappa \)-parametrized Whitehead, then every subgroup of \( A \) is weakly \( 2^\kappa \)-parametrized Whitehead.
    \begin{proof}
        For any abelian group \( A \), there is a canonical free resolution 
        \[ \begin{tikzcd}
            0 & K_A & F_A & A & 0
            \arrow[from=1-1, to=1-2]
            \arrow[from=1-2, to=1-3]
            \arrow["\pi_A", from=1-3, to=1-4]
            \arrow[from=1-4, to=1-5]
        \end{tikzcd} \]
        where \( F_A \) is the free abelian group generated by \( A \), and \( \pi_A : F_A \to A \) is the natural projection map with kernel \( K_A \). 
        
        Let \( B \) be a subgroup of \( A \). We identify \( F_B \) as a subgroup of \( F_A \), and note that \( K_B = K_A \cap F_B \). Consider the map \( f : K_A \to F_A/F_B \) given by \( f(a) = a + F_B \). Note that \( \ker(f) = K_B \), so it induces an injection \( K_A/K_B \to F_A/F_B \). Since \( F_A/F_B \cong F_{A \setminus B} \) is free, so is \( K_A/K_B \). It follows that there is a retraction \( r : K_A \to K_B \) with \( r \restrict K_B = \id \).
        
        Assume \( A \) is weakly \( 2^\kappa \)-parametrized Whitehead. We show that the same is true for \( B \). Let \( \varphi : 2^\kappa \to \Hom(K_B,\Z) \) be continuous. Define \( \Phi : 2^\kappa \to \Hom(K_A,\Z) \) by \( \Phi(x) = \varphi(x) \circ r \). Let \( \Psi : 2^\lambda \to \PHom(F_A,\Z) \) be a full continuous almost extension of \( \Phi \). Let \( \psi : 2^\lambda \to \PHom(F_B,\Z) \) be defined by \( \psi(x) = \Psi(x) \restrict F_B \). Then \( \psi \) is full continuous since \( \Psi \) is. For \( b \in K_B \), if \( \Psi(x)(b) = \Phi(x \restrict \kappa)(b) \), then 
        \[ \psi(x)(b) = \Psi(x)(b) = \Phi(x \restrict \kappa)(b) = \varphi(x \restrict \kappa)(r(b)) = \varphi(x \restrict \kappa)(b). \]
    \end{proof}
\end{lemma}

\begin{lemma}\label{lem:countable_whitehead}
    Let \( A \) be a countable abelian group. If \( A \) is weakly \( 2^\kappa \)-parametrized Whitehead, then \( A \) is Whitehead and thus free.
    \begin{proof}
        Let \( 0 \to K \to F \to A \to 0 \) be a free resolution of \( A \) such that \( F \) is countable. Let \( f \in \Hom(K,\Z) \) be given. Define \( \varphi : 2^\kappa \to \Hom(K,\Z) \) by \( \varphi(x) = f \) for all \( x \in 2^\kappa \). Let \( \psi : 2^\lambda \to \PHom(F,\Z) \) be an almost extension of \( \varphi \). For each \( a \in K \), there is an open dense set \( D_a \subseteq 2^\lambda \) such that for all \( x \in D_a \), \( \psi(x)(a) = \varphi(x \restrict \kappa)(a) = f(a) \). For \( a \in F \setminus K \), there is an open dense set \( D_a \subseteq 2^\lambda \) such that for all \( x \in D_a \), \( \psi(x)(a) \) is defined. By the Baire category theorem, there is \( x \in \bigcap_{a \in F} D_a \). Then \( \psi(x)(a) \) is defined for all \( a \in F \), so \( \psi(x) \in \Hom(F,\Z) \) is an extension of \( f \). 
    \end{proof}
\end{lemma}

\subsection{Forced Whiteheadness}\label{sec:forced-whiteheadness}

In this subsection, we characterize the property of being forced to be Whitehead in terms of parametrized homomorphisms, which we will apply to prove Corollary \ref{cor:forcing-characterization}. This section is not needed for understanding the proof of Theorem \ref{thm1}, and may be skipped by readers unfamiliar with forcing.

We show that homomorphisms parametrized by \( S(\P) \) correspond to a \( \P \)-name for a homomorphism. We construct the correspondence maps, but leave the verification of the properties to the readers.

\begin{theorem}[Second Correspondence Theorem]\label{thm:correspondence}
    Let \( \P \) be a separative poset. Then there is a correspondence 
    \[\begin{tikzcd}
        \{ \text{full continuous maps \( \varphi : S(\P) \to \PHom(A,B) \)} \} \\
        \{ \text{\( \P \)-names \( \dot{\varphi} \) such that \( \P \Vdash \dot{\varphi} \in \Hom(A,B) \)} \}
        \arrow[shift right, from=1-1, to=2-1]
        \arrow[shift right, from=2-1, to=1-1]
    \end{tikzcd}\]
    such that if \( \varphi \) corresponds to \( \dot{\varphi} \), then for \( a \in A \), \( b \in B \), \( x \in \dom(\varphi,a) \), and \( p \in \P \), we have
    \begin{align*}
        \varphi(x)(a) = b &\iff \exists p \in x(p \Vdash \dot{\varphi}(\check{a}) = \check{b}), \\
        \varphi[N_p](a) = b &\iff p \Vdash \dot{\varphi}(\check{a}) = \check{b}.
    \end{align*}
    Moreover, if \( \psi \) corresponds to \( \dot{\psi} \), then 
    \[ \varphi =^* \psi \iff \P \Vdash \dot{\varphi} = \dot{\psi}. \]
    \begin{proof}[Proof sketch]
        Let \( \varphi \) be given. Define 
        \[ \dot{\varphi} = \set{((a,b)\check{\vrule height1.3ex width0pt},p)}{a \in A \land b \in B \land p \in \P \land \forall x \in N_p (\varphi(x)(a) = b)}. \]
        Conversely, let \( \dot{\varphi} \) be given. Define \( \varphi : S(\P) \to \PHom(A,B) \) by
        \[ \varphi(x)(a) = b \iff \exists p \in x(p \Vdash \dot{\varphi}(\check{a}) = \check{b}). \]
        Note that \( \varphi \) is full since densely many conditions decide \( \dot{\varphi}(\check{a}) \).
    \end{proof}
\end{theorem}

\begin{corollary}
    Let \( K \) be a subgroup of \( F \). Let \( \P \) be a separative poset. Let \( \varphi : S(\P) \to \PHom(K,B) \) and \( \psi : S(\P) \to \PHom(F,B) \) be full continuous maps. Using the notation in the correspondence theorem, we have
    \[ \text{\( \psi \) almost extends \( \varphi \)} \iff \P \Vdash \text{\( \dot{\psi} \) extends \( \dot{\varphi} \)}. \]
    \qed
\end{corollary}

\begin{theorem}
    Let \( A \) be an abelian group with free resolution \( 0 \to K \to F \to A \to 0 \). Let \( \P \) be a separative poset and let \( \B \) be its Boolean completion. Then the following are equivalent:
    \begin{enumerate}
        \item \( A \) is forced by \( \P \) to be Whitehead;
        \item Every full continuous map \( \varphi : S(\P) \to \PHom(K,\Z) \) almost extends to some full continuous map \( \psi : S(\P) \to \PHom(F,\Z) \).
    \end{enumerate} \qed
\end{theorem}

\begin{corollary}\label{cor:forced_whitehead_implies_weakly_parametrized}
    Let \( S_\kappa \) be the Stone space of the Boolean completion of \( \P_\kappa \).
    If \( A \) is forced by \( \P_\kappa \) to be Whitehead, then \( A \) is weakly \( 2^\kappa \)-parametrized Whitehead. \qed
\end{corollary}

\section{Filtrations}\label{sec:filtration}

A key idea in proving Theorem \ref{thm1} is to write the group as a union of subgroups of smaller cardinality.

\begin{definition}
    Let \( \kappa \) be a cardinal. A \emph{\( \kappa \)-filtration} of an abelian group \( A \) is a sequence \( \ab{A_\alpha : \alpha < \kappa} \) of subgroups of \( A \) such that 
    \begin{enumerate}
        \item \( |A_\alpha| < \kappa \) for all \( \alpha < \kappa \), and
        \item \( A_\alpha \subseteq A_\beta \) for all \( \alpha < \beta < \kappa \), and
        \item \( A_\beta = \bigcup_{\alpha < \beta} A_\alpha \) for all limit ordinal \( \beta < \kappa \), and
        \item \( A = \bigcup_{\alpha < \kappa} A_\alpha \).
    \end{enumerate}
\end{definition}

It is not hard to show that if \( \kappa > \omega \) is regular, then every abelian group of size \( \kappa \) has a \( \kappa \)-filtration.

\begin{lemma}\label{lem2}
    Let \( \kappa > \omega \) be regular. Let \( \ab{A_\alpha : \alpha < \kappa} \) be a \( \kappa \)-filtration of \( A \). Suppose there is a club \( C \subseteq \kappa \) such that \( A_{\min(C)} \) is free, and \( A_\beta/A_\alpha \) is free for all \( \alpha < \beta \) in \( C \). Then \( A \) is free.
    \begin{proof}
        Let \( \ab{\gamma_\alpha : \alpha < \kappa} \) be the increasing enumeration of \( C \). Let \( B_0 \) be a basis of \( A_{\gamma_0} = A_{\min(C)} \). Inductively, assume \( B_\alpha \) is a basis of \( A_{\gamma_\alpha} \). Then since \( A_{\gamma_\alpha} \) and \( A_{\gamma_{\alpha+1}}/A_{\gamma_\alpha} \) are free, \( B_\alpha \) can be extended to a basis \( B_{\alpha+1} \) of \( A_{\gamma_{\alpha+1}} \). If \( \beta \) is a limit ordinal, then \( B_\beta = \bigcup_{\alpha < \beta} B_\alpha \) is a basis of \( A_{\gamma_\beta} = \bigcup_{\alpha < \beta} A_{\gamma_\alpha} \). Then \( B = \bigcup_{\alpha < \kappa} B_\alpha \) is a basis of \( A \), and thus \( A \) is free. 
    \end{proof}
\end{lemma}

Recall that an abelian group \( A \) is \emph{\( \kappa \)-free} if every subgroup of \( A \) of size less than \( \kappa \) is free.

\begin{lemma}
    Let \( \kappa > \omega \) be regular. Let \( A \) be a \( \kappa \)-free abelian group of size \( \kappa \) that is not free. Then there is a \( \kappa \)-filtration \( \ab{A_\alpha : \alpha < \kappa} \) of \( A \) such that \( S = \set{\alpha < \kappa}{A_{\alpha+1}/A_\alpha \text{ is not free}} \) is stationary.
    \begin{proof}
        Let \( \ab{B_\alpha : \alpha < \kappa} \) be any \( \kappa \)-filtration of \( A \). Let 
        \[ S = \set{\alpha < \kappa}{\exists \beta > \alpha, B_\beta/B_\alpha \text{ is not free}}. \]
        We show that \( S \) is stationary. If not, then there is a club \( C \subseteq \kappa \) such that \( B_\beta / B_\alpha \) is free for all \( \alpha < \beta \) in \( C \). Moreover, \( B_{\min(C)} \) is free since \( A \) is \( \kappa \)-free. Then \( A \) is free by Lemma \ref{lem2}, which is a contradiction.

        For each \( \alpha \in S \), let \( \beta_\alpha > \alpha \) be such that \( B_{\beta_\alpha}/B_\alpha \) is not free. Let 
        \[ D = \set{\gamma < \kappa}{\forall \alpha \in S \cap \gamma, \beta_\alpha < \gamma}. \]
        Then \( D \) is a club. Let \( \set{\gamma_\xi}{\xi < \kappa} \) be the increasing enumeration of \( D \).  Note that if \( \gamma \in D \) and \( \alpha \in S \cap \gamma \), then \( B_{\gamma}/B_{\alpha} \) is not free since it contains the non-free subgroup \( B_{\beta_\alpha}/B_\alpha \). In particular, if \( \gamma_\xi \in S \), then \( B_{\gamma_{\xi+1}}/B_{\gamma_\xi} \) is not free.
        
        Let \( A_\xi = B_{\gamma_\xi} \) for each \( \xi < \kappa \). Let \( E' = \set{\xi}{\gamma_\xi \in S} \). Then \( A_{\xi+1}/A_\xi \) is not free for all \( \xi \in E' \). Moreover, \( E' \) is stationary since it is the preimage of a stationary set under an increasing continuous function.
    \end{proof}
\end{lemma}

By \cite[\nopp XII.1.4]{EklofMekler1990}, a filtrated group has a filtrated free resolution:

\begin{lemma}\label{lem:filtrated_free_resolution}
    Let \( \kappa > \omega \) be regular. Let \( A \) be an abelian group with \( \kappa \)-filtration \( \ab{A_\alpha : \alpha < \kappa} \). Then there is a free resolution \( 0 \to K \to F \to A \to 0 \) with the following properties: 
    \begin{enumerate}
        \item \( K = \bigoplus_{\beta<\kappa} K_\beta^* \) and \( F = \bigoplus_{\beta<\kappa} F_\beta^* \) are direct sums of free abelian groups.
        \item Let \( K_\alpha = \bigoplus_{\beta < \alpha} K_\beta^* \) and \( F_\alpha = \bigoplus_{\beta < \alpha} F_\beta^* \). Then \( \ab{K_\alpha : \alpha < \kappa} \) and \( \ab{F_\alpha : \alpha < \kappa} \) are \( \kappa \)-filtrations of \( K \) and \( F \), respectively.
        \item For all \( \alpha < \kappa \), there is a commutative diagram
        \[ \begin{tikzcd}
        0 \arrow[r] & K_\alpha \arrow[r] \arrow[d] & F_\alpha \arrow[r] \arrow[d] & A_\alpha \arrow[r] \arrow[d] & 0 \\
        0 \arrow[r] & K \arrow[r] & F \arrow[r] & A \arrow[r] & 0
    \end{tikzcd} \]
    where the rows are free resolutions and the vertical maps are inclusions. \qed
    \end{enumerate}
\end{lemma}

\section{Proof of the Main Theorem}\label{sec:main-theorem}

We first prove some lemmas.

Recall that a poset \( \P \) has the \emph{countable chain condition (ccc)} if every antichain in \( \P \) is countable. It is well known that \( \P_\kappa = \Add(\omega,\kappa) \) has the ccc \cite[Lemma III.3.7]{Kunen2013Set}.

\begin{lemma}\label{lemma:full-continuous}
    If \( \varphi : 2^\kappa \to \PHom(A,B) \) is full continuous, then for each \( a \in A \), there is an open dense \( D \subseteq 2^\kappa \) and a countable set \( \sigma \subseteq \kappa \) such that for all \( x,y \in D \), if \( x \restrict \sigma = y \restrict \sigma \), then \( \varphi(x)(a) \) and \( \varphi(y)(a) \) are both defined and are equal.
    \begin{proof}
        Since \( \varphi \) is full, \( \mathbb{D} = \set{p \in \P_\kappa}{\forall x \in N_p (\varphi(x)(a) \text{ is defined})} \) is open dense. Choose a maximal antichain \( \mathbb{A} \subseteq \mathbb{D} \) such that for all \( p \in \mathbb{A} \), there is \( b \in B \) such that \( \varphi(x)(a) = b \) for all \( x \in 2^\kappa \) extending \( p \). Let \( D = \set{x \in 2^\kappa}{\exists p \in \mathbb{A}(p \subseteq x)} \). Then \( D \) is open since it is the union of the basic open sets \( N_p \) for \( p \in \mathbb{A} \). We show that \( D \) is dense. Let \( N_p \) be any basic open set. Then \( p \) is compatible with some \( q \in \mathbb{A} \). Let \( r \leq p,q \). Then \( N_r \subseteq N_p \cap N_q \subseteq N_p \cap D \).

        Let \( \sigma = \bigcup_{p \in \mathbb{A}} \dom(p) \). Since \( \P_\kappa \) is ccc, \( \mathbb{A} \) is countable and thus \( \sigma \) is countable. Let \( x,y \in D \) with \( x \restrict \sigma = y \restrict \sigma \). Since \( \mathbb{A} \) is an antichain, there is \( p \in \mathbb{A} \) such that \( p \subseteq x \) and \( p \subseteq y \). Let \( b \in B \) be such that \( \varphi(z)(a) = b \) for all \( z \in 2^\kappa \) extending \( p \). Then \( \varphi(x)(a) = \varphi(y)(a) = b \).
    \end{proof}
\end{lemma}

\begin{lemma}[See {\cite[\nopp XII.1.4]{EklofMekler1990}}]\label{lemma:not-weakly-parametrized}
    Let \( F \) be a free abelian group and let \( K,H \subseteq F \) be subgroups. Suppose \( F/H \) is free and \( F/(H+K) \) is not weakly \( 2^\kappa \)-parametrized Whitehead. Then there are continuous maps \( f_0,f_1 : 2^\kappa \to \Hom(K,\Z) \) such that for all \( \lambda \geq \kappa \), there are no full continuous almost extensions \( h_0,h_1 : 2^\lambda \to \PHom(F,\Z) \) such that \( h_0 =^* h_1 \) on \( H \).
    \begin{proof}
        Let \( f_0(x) = 0 \) for all \( x \in 2^\kappa \).

        Since \( F/(H+K) \cong (F/H)/((H+K)/H) \) is not weakly \( 2^\kappa \)-parametrized Whitehead, there is a continuous \( g : 2^\kappa \to \Hom((H+K)/H,\Z) \) with no full continuous almost extension \( 2^\lambda \to \PHom(F/H,\Z)\). Let \( f_1(x)(k) =g(x)(k+H)\).

        Suppose towards a contradiction that there are \( h_0,h_1 : 2^\lambda \to \PHom(F,\Z) \) as stated above. Let \( h' = h_1-h_0 \). Then \( h' =^* 0 \) on \( H \). Since \( F/H \) is free, there is a homomorphism \( s : F/H \to F \) such that \( \pi(s(c)) = c \) for all \( c \in F/H \). For \( x \in 2^\lambda \) and \( c \in F/H \), let \( g'(x)(c) = h'(x)(s(c)) \). It is not hard to check that \( g' : 2^\lambda \to \PHom(F/H,\Z) \) is a full continuous map. 

        We show that \( g' \) almost extends \( g \). Fix \( k \in K \). Then \( s(k+H) \in k+H \), so \( s(k+H) = k+a \) for some \( a \in H \). There is an open dense set of \( x \in 2^\lambda \) such that \( h'(x)(a) = 0 \) and \( h_i(x)(k) = f_i(x \restrict \kappa)(k) \) for \( i = 0,1 \). For such \( x \), we have
        \[ g'(x)(k+H) = h'(x)(k+a) = h'(x)(k) = f_1(x \restrict \kappa)(k) = g(x \restrict \kappa)(k+H), \]
        which contradicts the choice of \( g \).
    \end{proof}
\end{lemma}

We introduce some notations for manipulating binary sequences. 

For \( x,y \in 2^\kappa \), define \( x \oplus y \in 2^\kappa \) as follows. For each limit ordinal \( \alpha < \kappa \) and \( n < \omega \), let \( (x \oplus y)(\alpha + 2n) = x(\alpha+n) \) and \(  (x \oplus y)(\alpha + 2n+1) = y(\alpha+n) \). Note that if \( \alpha \) is limit, then \( (x \oplus y)(\alpha) = x(\alpha) \) and \( (x \oplus y) \restrict \alpha = x \restrict \alpha \oplus y \restrict \alpha \). Moreover, the map \( \oplus : 2^\kappa \times 2^\kappa \to 2^\kappa \) is a homeomorphism. 

For \( x \in 2^\kappa \), define \( x_\alpha^i \in 2^\kappa \) by \( x_\alpha^i(\alpha) = i \) and \( x_\alpha^i(\beta) = x(\beta) \) for \( \beta \neq \alpha \).

The following lemma is straightforward to verify.

\begin{lemma}\label{lem:open_continuous_maps}
    Let \( x,y \in 2^\kappa \). Let \( \sigma, \tau \subseteq \kappa \) be disjoint. Let \( \rho : \gamma \to \kappa \) be an injection. Then the following maps are open and continuous: 
    \[ \setlength{\tabcolsep}{18pt}
    \begin{tabular}{cccc}
        \( 2^\kappa \to 2^{\sigma} \times 2^{\tau} \) & \( 2^\kappa \to 2^\gamma \) & \( 2^\kappa \times 2^\kappa \to 2^\kappa \) & \( 2^\kappa \to 2^\kappa \) \\
        \( x \mapsto (x \restrict \sigma, x \restrict \tau) \) & \( x \mapsto x \circ \rho \) & \( (x,y) \mapsto x \oplus y \) & \( x \mapsto x_\alpha^i \)
    \end{tabular} \]
    \qed
\end{lemma}

\begin{lemma}\label{lem:fullness-preserved}
    If \( \varphi : X \to \PHom(A,B) \) is full and \( \pi : Y \to X \) is an open continuous map, then \( \varphi \circ \pi \) is full.
    \begin{proof}
        The preimage of an open dense set under an open continuous map is open dense.
    \end{proof}
\end{lemma}

The following theorem together with Lemma \ref{lem:condensed_whitehead_implies_weakly_parametrized} immediately implies Theorem \ref{thm1}, and together with Corollary \ref{cor:forced_whitehead_implies_weakly_parametrized} implies Corollary \ref{cor:forcing-characterization}.

\begin{theorem}\label{thm:main}
    Let \( A \) be an abelian group of size \( \kappa \). If \( A \) is weakly \( 2^\kappa \)-parametrized Whitehead, then \( A \) is free.
    \begin{proof}
        By induction on \( \kappa \), assume the theorem holds for groups of size less than \( \kappa \). 
        
        If \( B \) is a subgroup of \( A \) with \( |B| = \gamma < \kappa \), then \( B \) is weakly \( 2^\kappa \)-parametrized Whitehead and thus weakly \( 2^\gamma \)-parametrized Whitehead. So, \( B \) is free by the induction hypothesis. Hence, \( A \) is \( \kappa \)-free. If \( \kappa \) is singular, then \( A \) is free by Shelah's singular compactness theorem \cite{ShelahSingularCompactness}. So, we may assume \( \kappa \) is regular. Moreover, since countable weakly \( 2^\kappa \)-parametrized Whitehead groups are free, we may assume \( \kappa \) is uncountable. 

        Let \( \ab{A_\alpha : \alpha < \kappa} \) be a \( \kappa \)-filtration of \( A \) such that \( S = \set{\alpha < \kappa}{A_\alpha^* \text{ is not free}} \) is stationary, where \( A_\alpha^* = A_{\alpha+1}/A_\alpha \). Let \( 0 \to K \to F \to A \to 0 \) be the corresponding filtrated free resolution of \( A \) given by Lemma \ref{lem:filtrated_free_resolution}. 

        For \( \alpha \in S \), \( A_\alpha^* \) is not weakly \( 2^\kappa \)-parametrized Whitehead by the induction hypothesis and Lemma \ref{lem:weakly_parametrized_implies_smaller_weakly_parametrized}. Note that \( A_\alpha^* \cong F_{\alpha+1}/(F_\alpha + K_\alpha^*) \). By Lemma \ref{lemma:not-weakly-parametrized}, there are continuous maps \( f_{0\alpha}, f_{1\alpha} : 2^\kappa \to \Hom(K_\alpha^*,\Z) \) such that for all \( \lambda \geq \kappa \), there are no full continuous almost extensions \( h_0,h_1 : 2^\lambda \to \PHom(F_{\alpha+1},\Z) \) such that \( h_0 =^* h_1 \) on \( F_\alpha \). For \( \alpha \notin S \), we take \( f_{0\alpha} = f_{1\alpha} = 0 \).

        Let \( \varphi : 2^\kappa \to \Hom(K,\Z) \) be given by 
        \[ \varphi(x \oplus y) = \bigoplus_{\alpha < \kappa} f_{x(\alpha)\alpha}(y). \]
        It is easy to see that \( \varphi \) is continuous. Since \( A \) is weakly \( 2^\kappa \)-parametrized Whitehead, there is a full continuous almost extension \( \psi : 2^\lambda \to \PHom(F,\Z) \) of \( \varphi \). By Lemma \ref{lemma:full-continuous}, for each \( a \in F \), there is an open dense \( D_a \subseteq 2^\lambda \) and a countable set \( \sigma_a \subseteq \kappa \) such that for \( z \in D_a \), \( \psi(z)(a) \) is defined and only depends on \( z \restrict \sigma_a \) and \( z \restrict (\lambda \setminus \kappa) \). We may also assume that \( \psi(z)(a) = \varphi(z \restrict \kappa)(a) \) for all \( a \in K \) and \( z \in D_a \). 
        
        Let \( C = \set{\alpha < \kappa}{\forall a \in F_\alpha (\sigma_a \subseteq \alpha)} \). Then \( C \) is a club. Since \( S \) is stationary, there is a limit ordinal \( \alpha \in S \cap C \).

        Fix an injection \( \rho : \lambda \to \lambda^+ \setminus \lambda \). For \( i \in \{ 0,1 \} \), let \( r_i : 2^{\lambda^+} \to 2^\lambda \) be defined by
        \[ r_i(y) = (y \circ \rho)_\alpha^i \oplus (y \restrict \lambda). \]
        Let \( h_i : 2^{\lambda^+} \to \PHom(F_{\alpha+1},\Z) \) be given by 
        \[ h_i(y) = \psi(r_i(y)) \restrict F_{\alpha+1}. \]
        By Lemma \ref{lem:open_continuous_maps}, the map \( r_i \) is open and continuous, being a composition of open continuous maps. Let \( C_a = r_0^{-1}(D_a) \cap r_1^{-1}(D_a) \). Then \( C_a \) is open dense in \( 2^{\lambda^+} \). For \( a \in F_{\alpha+1} \) and \( y \in C_a \), \( h_i(y)(a) \) is defined. For \( a \in F_\alpha \) and \( y \in C_a \), since \( r_0(y) \) and \( r_1(y) \) agree on \( \sigma_a \subseteq \alpha \) and \( \lambda \setminus \kappa \), we have 
        \[ h_0(y)(a) = \psi(r_0(y))(a) = \psi(r_1(y))(a) = h_1(y)(a). \]
        For \( a \in K_\alpha^* \) and \( y \in C_a \), 
        \[ h_i(y)(a) = \psi((y \circ \rho)_\alpha^i \oplus (y \restrict \lambda))(a) = \varphi((y \circ \rho)_\alpha^i \restrict \kappa \oplus (y \restrict \kappa))(a) = f_{i\alpha}(y \restrict \kappa)(a). \]
        Hence, \( h_i \) is a full continuous almost extension of \( f_{i\alpha} \) with \( h_0 =^* h_1 \) on \( F_\alpha \). This contradicts the choice of \( f_{0\alpha} \) and \( f_{1\alpha} \).
    \end{proof}
\end{theorem}

\printbibliography

\end{document}